\documentclass[12pt,reqno]{amsart}

\usepackage{amsmath,amsfonts,amsthm,amssymb,amsxtra}
\usepackage{hyperref}
\usepackage{bbm} 
\usepackage{bm} 
\usepackage{color}
\usepackage{accents} 
\makeindex 

\newtheorem{theorem}{Theorem}[section]

\newtheorem{lemma}[theorem]{Lemma}
\newtheorem{corollary}[theorem]{Corollary}

\theoremstyle{definition}

\theoremstyle{remark}

\newtheorem{remark}[theorem]{Remark}

\numberwithin{equation}{section}

\renewcommand{\epsilon}{\varepsilon}

\renewcommand{\phi}{\varphi}

\makeatletter
\newcommand*{\rom}[1]{\expandafter\@slowromancap\romannumeral #1@}
\makeatother

\begin{document}

\title[The $p$-Laplacian overdetermined problem on Riemannian manifolds ]{The $p$-Laplacian overdetermined problem on Riemannian manifolds}

\author[Q. Ruan]{Qihua Ruan}
\address[Qihua Ruan]{Key Laboratory of Financial Mathematics of Fujian Province University, Putian University, Fujian Putian, 351100, P.R., China}
\email{ruanqihua@163.com}

\author[Q. Huang]{Qin Huang}
\address[Qin Huang]{Fujian Key Laboratory of Financial Information Processing, Putian University, Fujian Putian, 351100, P.R., China}
\email{qinhuang78@163.com}

\author[F. Chen]{Fan Chen*}
\address[Fan Chen]{School of Mathematics and Statistics, Fujian Normal University, Qishan Campus, Fuzhou, 350117, P.R., China;  Fujian Key Laboratory of Financial Information Processing, Putian University, Fujian Putian, 351100, P.R., China}
\email{chenfan@163.com}




\begin{abstract}
  In this paper, we study the overdetermined problem for the $p$-Laplacian equation on complete noncompact Riemannian manifolds  with nonnegative Ricci curvature. We prove that the regularity results of weak solutions of the $p$-Laplacian equation and obtain some integral identities. As their applications, we give the proof of the $p$-Laplacian overdetermined problem and obtain some well known results such as the Heintze-Karcher inequality and the Soap Bubble Theorem.
\end{abstract}

\thanks{This work was supported by NSF of China(No. 11971253) and NSF of Fujian Province(No. 2021J011101).}


\maketitle
\setcounter{tocdepth}{1}

\section{Introduction and main result}\label{Introduction and main results}

The celebrated Serrin's \cite{Serrin[1]} theorem states that the overdetermined boundary value problem
\begin{equation}\label{Serrin}
\begin{cases}
-\Delta u=1 & \text{in $\Omega$},\\
u=0   & \text{on $\partial\Omega$},\\
u_{\nu}=c & \text{on $\partial\Omega$},
\end{cases}
\end{equation}
admits a solution for some positive constant $c$ if and only if $\Omega$ is a ball of radius $nc$ and $u(x)=\frac{n^{2}c^{2}-\vert x\vert^{2}}{2n}$. Here, $\Omega$ denotes a bounded domain in $\mathbb{R}^{n}(n\geq2)$, with
 a smooth boundary $\partial\Omega$ and $u_{\nu}$ is the outward normal derivative of $u$ on $\partial\Omega$. The Serrin's proof method is now known as the method of the moving plane. Later, Weinberger \cite{WH} used some integral identities and some basic inequalities to give a very simple proof. This result has its physical explanation, the associated Dirichlet problem describes a viscous incompressible fluid moving in straight parallel streamlines through a straight pipe of given cross sectional form $\Omega$, more detail see \cite{Ros and Sicbaldi[2],S1956}.

Serrin's symmetry result has been extensively studied and generalized also to the case of quasilinear elliptic problems. The $p$-Laplacian is defined by
\begin{eqnarray}\label{p-Laplacian}
\Delta_{p}u={\rm div}(\lvert\nabla u\rvert^{p-2}\nabla u),
\end{eqnarray}
for $p>1$.
It has been proved that if the following overdetermined problem
\begin{equation}\label{p-Laplacian overdetermined problem}
\begin{cases}
-\Delta_{p} u=1 & \text{in $\Omega$},\\
u=0   & \text{on $\partial\Omega$},\\
u_{\nu}=c & \text{on $\partial\Omega$},
\end{cases}
\end{equation}
admits a weak solution in the bounded domain $\Omega\subset\mathbb{R}^{n}$ , then $\Omega$ is a ball. The overdetermined problem \eqref{p-Laplacian overdetermined problem} was first proved by Garofalo and Lewis \cite{GL} through using Weinberger's method. In 2000, Damascelli and Pacella \cite{DP} proved this result via Serrin's approach in the case of $1<p<2$. Later,  Brock and Henrot \cite{BH} proposed a different proof of the result via Steiner symmetrization for $p\geq2$. The literature about overdetermined problems of the elliptic equations is so wide that it is impossible to report it exhaustively. We refer the interested reader to \cite{Cianchi[18],Bianchini and Ciraolo[29],Wang and Xia[23]} and the references therein.

In 2020, Colasuonno and Ferrari \cite{FF} considered the following Dirichlet $p$-Laplacian problem
\begin{equation}\label{*}
\begin{cases}
-\Delta_{p} u=1 & \text{in $\Omega$},\\
u=0   & \text{on $\partial\Omega$},
\end{cases}
\end{equation}
in $\mathbb{R}^{n}$. With use of some integral identities, they prove the equivalence of the following Soap Bubble Theorem and the Serrin-type symmetry result for the overdetermined problem \eqref{p-Laplacian overdetermined problem}, for $1<p<2$.

\begin{theorem}{\rm(\cite{FF})}\label{CR1.1}
Let $\Omega\subset\mathbb{R}^{n}$ be a bounded domain with boundary $\partial\Omega$ of class $C^{2,\alpha}$, and denote by $H=H(x)$ the mean curvature of $\partial\Omega$. Suppose that $u$ solves \eqref{*} for $1<p<2$, and the set of critical points of $u$ has zero measure. Then the following statements are equivalent:

{\rm(A)} $\Omega$ is a ball;

{\rm(B)} $\vert u_{\nu}(x)\vert^{p-2}u_{\nu}(x)=-\frac{1}{nH(x)}$ for every $x\in\partial\Omega$;

{\rm(C)} $u$ is radial;

{\rm(D)} $H(x)=H_{0}$ for every $x\in\partial\Omega$ and some constant $H_{0}$;

Moreover, if one of the above conclusions holds, then

{\rm(E)} $\vert\nabla u(x)\vert=(\frac{1}{nH_{0}})^{\frac{1}{p-1}}$ for every $x\in\partial\Omega$.
\end{theorem}

According to Lou \cite{Lou}, also see Antonini etc.\cite{ACF}, we know that the critical set of $u$ in the above theorem  has zero  measure. So this assumption can be removed. For $p$-Laplacian operator, the normal condition of $p$ is that any $p>1$. However, Theorem \ref{CR1.1} only solves the problem of $1<p<2$. For $p\geq2$, Deng and Yin \cite{DY} were inspired by Valtorta's result \cite{VD} and proved the equivalent conclusion of the above theorem by employing a more refined inequality of the Cauchy Schwartz inequality, see the inequality \eqref{refined estimate}. Unfortunately, their method can not be used to the case of $1<p<2$.

In this paper, we want to find a method to prove the above theorem for any $p>1$. We also find that the equivalence of (A) and (B) in the above theorem also holds on a Riemannian manifold. We notice that the equivalence problem of (A) and (E) in the above theorem on a Riemannian manifold is an open problem until now. To our knowledge, there are only few results about the overdetermined problem on some special manifolds, such as constant curvature space, see \cite{AA,GV,FAW}.

Before stating the main results, let us make some remarks about some notations to be used. Let $(M,g)$ be an $n$-dimensional complete noncompact Riemannian manifold. We use $\langle,\rangle$ to denote the inner product with respect to both $g$ and $g_{\partial M}$ when no confusion occurs. We denote by $\nabla$, $\Delta$ and $\nabla^{2}$ the gradient, the Laplacian and the Hessian, respectively. Let $\rm dv$ and $\rm ds$ be the canonical volume element and area element of $M$ and $\partial M$. We shall use $u_{i}$ and $u_{ij}$ to denote, respectively, $\nabla_{i}u$ and $\nabla_{j}\nabla_{i}u$. Finally, repeated indices indicate a summation with respect to the index.

For convenience, we introduce the so-called $P$-function:
\begin{eqnarray}\label{P-function}
P:=\frac{p-1}{p}\vert\nabla u\vert^{p}+\frac{1}{n} u,
\end{eqnarray}
where $u$ is a solution of \eqref{*}.

In order to obtain the Soap Bubble Theorem and the Heintze-Karcher inequality, we need the following integral identity.
\begin{theorem}\label{Fundamental Identity Theorem}
Let $(M,g)$ be an $n$-dimensional complete noncompact Riemannian manifold. Assume that $\Omega\subset M$ be a bounded and connected domain with boundary $\partial \Omega$ of class $C^{2,\alpha}$, and $H=H(x)$ the mean curvature of $\partial\Omega$. Let $u\in W^{1,p}(\Omega)$ be a weak solution to the problem \eqref{*} for any $p>1$.  Then the following identity holds
\begin{eqnarray}\label{Fundamental Identity}
\frac{1}{(p-1)(n-1)}\int_{\Omega}\mathcal{L}_{u}P{\rm dv}=\frac{1}{n}\vert\Omega\vert-\int_{\partial\Omega}H\vert u_{\nu}\vert^{2p-2}{\rm ds}.
\end{eqnarray}
Here, $\mathcal{L}_{u}$ is a linearized operator, see \eqref{7.1}, and $\vert\Omega\vert$ denotes the volume of $\Omega$.
\end{theorem}

\begin{remark}
When $M=\mathbb{R}^{n}$, this integral identity was proved by Magnanini and Poggesi {\rm \cite{MP3}}$($for $p=2$$)$ and Deng and Yin {\rm \cite{DY}}$($for $p\geq2$$)$.
\end{remark}

Another expression of the integral identity \eqref{Fundamental Identity} is as follows.
\begin{theorem}\label{Soap Bubble Theorem and the Heintze-Karcher inequality}
Under the conditions of Theorem {\rm \ref{Fundamental Identity Theorem}}. Then the following identity holds
\begin{eqnarray}\label{SB and HK}
\frac{n^{2}}{(p-1)(n-1)} \int_{\Omega} \mathcal{L}_{u} P {\rm dv}&+&\int_{\partial \Omega} \frac{1}{H}\left(1+nH\vert u_{\nu}\vert^{p-2} u_{\nu}\right)^{2} {\rm ds} \\
&&=\int_{\partial \Omega} \frac{1}{H}{\rm ds}-n\vert\Omega\vert.\nonumber
\end{eqnarray}
\end{theorem}

The Heintze-Karcher inequality dates back to 1987, and was obtained by Ros {\rm \cite{Ros}} using Reilly's Formula {\rm \cite{Re1}}. In 2014, Huang and Ruan {\rm \cite{HR}} used the divergence theorem to extend the Ros's result to manifolds with nonnegative Bakry-Emery Ricci curvature. Deng and Yin also obtained the Heintze-Karcher inequality in $\mathbb{R}^{n}$, see {\rm \cite{DY}}.

As an application of the integral identity \eqref{SB and HK}, we deduce the following Heintze-Karcher inequality.
\begin{corollary}\label{Heintze-Karcher inequality theorem}
Let $(M,g)$ be an $n$-dimensional complete noncompact Riemannian manifold with nonnegative Ricci curvature. Assume that $\Omega\subset M$ is a bounded and connected domain with boundary $\partial \Omega$ of class $C^{2,\alpha}$. If the mean curvature $H$ is positive on $\partial\Omega$, then
\begin{eqnarray}\label{Heintze-Karcher inequality}
\int_{\partial \Omega} \frac{1}{H}{\rm ds} \geq n\vert\Omega\vert
\end{eqnarray}
holds. In particular, the equality holds if and only if $\Omega$ is an Euclidean ball and $M$ is isometric to $\mathbb{R}^{n}$.
\end{corollary}

As another application of the integral identity \eqref{SB and HK}, we  prove  the following $p$-Laplacian overdetermined problem.
\begin{theorem}\label{overdetermined problem theorem}
Let $(M,g)$ be an $n$-dimensional complete noncompact Riemannian manifold with nonnegative Ricci curvature. Assume that $\Omega\subset M$ is a bounded and connected domain with boundary $\partial \Omega$ of class $C^{2,\alpha}$, and $H=H(x)$ is the mean curvature of $\partial\Omega$. Let $u\in W^{1,p}(\Omega)$ be a weak solution to the problem \eqref{*} for any $p>1$. Then $u_{\nu}$ on $\partial\Omega$ satisfies that $u_{\nu}\vert u_{\nu}\vert^{p-2}=-\frac{1}{nH}$ if and only if $\Omega$ is an Euclidean ball and $M$ is isometric to $\mathbb{R}^{n}$.
\end{theorem}

The integral identity \eqref{Fundamental Identity} can also be rewritten to the following form.
\begin{theorem}\label{Soap Bubble Theorem}
Under the conditions of Theorem {\rm \ref{Fundamental Identity Theorem}}. Then the following identity holds
\begin{eqnarray}\label{SBT}
\frac{1}{(p-1)(n-1)}\int_{\Omega}\mathcal{L}_{u}P{\rm dv}&+&\frac{1}{n^{2}H_{0}}\int_{\partial\Omega}(nu_{\nu}\vert u_{\nu}\vert^{p-2}H_{0}+1)^{2}{\rm ds}\\
&&=\int_{\partial\Omega}(H_{0}-H)\vert u_{\nu}\vert^{2p-2}{\rm ds}.\nonumber
\end{eqnarray}
Here, $H_{0}=\frac{\vert\partial \Omega\vert}{n\vert\Omega\vert}$, and $\vert\partial\Omega\vert$ denotes the area of $\partial\Omega$.
\end{theorem}

The celebrated Alexandrov's Soap Bubble Theorem states that if the mean curvature $H$ of a compact hypersurface $\Gamma$ embedded in $\mathbb{R}^{n}$ is constant, then $\Gamma$ is a sphere.
To prove this result, Alexandrov introduced his reflection principle, see \cite{A12}. With help of some properties of the solution of the PDE  $\Delta u=-1$, Reilly gave an alternative proof of Alexandrov's theorem in \cite{Re1}.
Recently, Magnanini and Poggesi proved Alexandrov's Soap Bubble Theorem and the stability of Serrin's results in \cite{MP3,MP2} by using some integral identities proved in \cite{WH} and refined in \cite{PS}.

As an application of the integral identity \eqref{SBT}, we  obtain the following Soap Bubble Theorem.
\begin{corollary}\label{overdetermined problem theorem 2}
Let $(M,g)$ be an $n$-dimensional complete noncompact Riemannian manifold with nonnegative Ricci curvature. Assume that $\Omega\subset M$ is a bounded and connected domain with boundary $\partial \Omega$ of class $C^{2,\alpha}$. If the mean curvature $H= H_{0}$ holds on $\partial\Omega$, where $H_{0}=\frac{\vert\partial \Omega\vert}{n\vert\Omega\vert}$, then $\Omega$ is an Euclidean ball and $M$ is isometric to $\mathbb{R}^{n}$.
\end{corollary}

The key step in our proof is to show that the $P$-function satisfies the subharmonic property of the linearized $p$-Laplacian, namely, $\mathcal{L}_{u}P\geq0$, for $p>1$ in $\Omega\backslash\mathcal{C}$, where $\mathcal{C}:=\{x\in\Omega:\ \vert\nabla u(x)\vert=0\}$.
When we prove the estimate of $\mathcal{L}_{u}P\geq0$, we have to estimate the term of $\|\nabla^{2}u\|^{2}$, here $\|\nabla^{2}u\|$ denotes the Hilbert-Schmidt norm on matrices defined to be
\begin{eqnarray*}
\|\nabla^{2}u\|=(\sum \limits_{i,j} \vert u_{ij}\vert^{2})^{\frac{1}{2}}.
\end{eqnarray*}

Now let's briefly review the role of $\|\nabla^{2}u\|^{2}$. The following simple consequence of the Cauchy Schwartz inequality is well known,
\begin{eqnarray}\label{Cauchy Schwartz inequality}
\|\nabla^{2}{u}\|^{2} \geq \frac{(\Delta u)^{2}}{n}.
\end{eqnarray}
This inequality plays a crucial role in Relly's (\cite{Re1,Re2}) proof of Alexandrov's Theorem and Weinberger's (\cite{WH}) proof of Serrin's Theorem. In fact, when $u$ is a solution of \eqref{Serrin}, the equality in \eqref{Cauchy Schwartz inequality} holds if and only if $u$ is a quadratic polynomial $\omega$ of the form $\omega(x)=\frac{1}{2n}(\vert z-x\vert^{2})-a$, for some choice of $z\in \mathbb{R}^{n}$ and $a\in\mathbb{R}^{+}$. The boundary condition clearly implies that $\Omega$ must be a ball.

A more refined estimate on $\|\nabla^{2}{u}\|^{2}$ is the following
\begin{eqnarray*}
\|\nabla^{2}{u}\|^{2} \geq \frac{(\Delta u)^{2}}{n}+\frac{n}{n-1}\left(\frac{\Delta u}{n}-A_{u}\right)^{2},
\end{eqnarray*}
where $A_{u}=\vert\nabla u\vert^{-2}\nabla^{2}u(\nabla u, \nabla u)$. This estimate is the analogue of the curvature-dimension inequality and plays a key role in \cite{BQ} to prove the comparison with the one dimensional model.

In \cite{VD}, Valtorta generalized the above inequality to $p$-Laplacian and obtained the following estimate,
\begin{eqnarray}\label{refined estimate}
&&\vert\nabla u\vert^{2 (p-2)}\left(\|\nabla^{2}{u}\|^{2}+p(p-2) A_{u}^{2}\right) \\
&\geq& \frac{\left(\Delta_{p} u\right)^{2}}{n}+\frac{n}{n-1}\left(\frac{\Delta_{p} u}{n}-(p-1)\vert\nabla u\vert^{p-2} A_{u}\right)^{2}.\nonumber
\end{eqnarray}
He used the above inequality to prove a sharp estimate on the first nontrivial eigenvalue of the $p$-Laplacian on a compact Riemannian manifold with nonnegative Ricci curvature.

The inequality \eqref{refined estimate} is used by Deng and Yin \cite{DY} to prove $\mathcal{L}_{u}P\geq0$. However they obtain this results under the condition of $p\geq2$. In this paper, we want to remove the condition of $p\geq2$. We will modify the inequality \eqref{refined estimate} to the following new inequality:
\begin{eqnarray}\label{MPL1}
&&\vert\nabla u\vert^{2(p-2)}\left(\|\nabla^{2}{u}\|^{2}+(p^{2}-2p+2)A_{u}^{2}\right) \\
&\geq& \frac{\left(\Delta_{p} u\right)^{2}}{n}+\frac{n}{n-1}\left(\frac{\Delta_{p} u}{n}-(p-1)\vert\nabla u\vert^{p-2} A_{u}\right)^{2}
+2\vert\nabla u\vert^{2(p-2)}\vert\nabla \vert\nabla u\vert\vert^{2}. \nonumber
\end{eqnarray}
We compare the inequality \eqref{MPL1} with \eqref{refined estimate}. From the Cauchy Schwartz inequality $A^{2}_{u}\leq\vert\nabla\vert\nabla u\vert\vert^{2}$, the inequality \eqref{MPL1} implies \eqref{refined estimate}.

The paper is organized as follows. In Section \ref{Preliminaries}, we introduce some important quantities, some known results and some preliminary lemmas. In Section \ref{section3}, we prove the main results.

\section{Preliminaries}\label{Preliminaries}

\subsection{The $p$-Laplacian on non-critical level sets of $u$}

Let us denote by $H$ the mean curvature of $\partial \Omega$. To be more precise,
\begin{eqnarray}\label{6.1}
(n-1)H=\sum_{k=1}^{n-1}\langle\nabla_{e_{k}}\nu,e_{k}\rangle,
\end{eqnarray}
where $\nu$ is the unit outer normal vector field on $\partial \Omega$ and $\{e_{k}\}(k=1,2,...,n-1)$ are local orthonormal vector fields which are tangent to $\partial \Omega$ at points on the boundary.
Let $u$ be a solution of \eqref{*}, we denote by $\nu$ the following vector field
\begin{eqnarray*}
\nu=-\frac{\nabla u}{\vert\nabla u\vert},\ {\rm where}\ \nabla u\neq0,
\end{eqnarray*}
which coincides with the outward unit normal on $\partial\Omega$. Then, it is easy to see that the Laplacian of $u$ on $\partial \Omega$ can be expressed as follows
\begin{eqnarray}\label{6.2}
\Delta u=(n-1)Hu_{\nu}+u_{\nu\nu},
\end{eqnarray}
where $u_{\nu}=\nabla u\cdot\nu=-\vert\nabla u\vert$ and  $u_{\nu\nu}=\nabla^{2}u(\nu,\nu)$. Inasmuch as $\Delta_{p}u=-1$, we also get that
\begin{eqnarray}\label{6.3}
\vert u_{\nu}\vert^{p-2}\left[\Delta u+(p-2) u_{\nu\nu}\right]=-1.
\end{eqnarray}
which implies together with \eqref{6.2}
\begin{eqnarray}\label{6.4}
\vert u_{\nu}\vert^{p-2}\left[(p-1) u_{\nu\nu}+(n-1) H u_{\nu}\right]=-1 \ {\text on} \ \partial \Omega.
\end{eqnarray}

\subsection{Linearized $p$-Laplacian and some lemmas}

In this subsection we introduce the linearized operator of the $p$-Laplacian. For a function $\eta$ of class $C^{2}$, we define
\begin{eqnarray}\label{7.1}
\mathcal{L}_{u}\eta &\equiv &\frac{{\rm d}}{{\rm d}t}{\Big\vert}_{t=0} \Delta_{p}(u+t \eta) \\
&=& \operatorname{{\rm div}}\left((p-2)\vert\nabla u\vert^{p-4}\langle\nabla u,\nabla \eta\rangle \nabla u+\vert\nabla u\vert^{p-2} \nabla \eta\right) \nonumber \\
&=&(p-2) \frac{\langle\nabla u,\nabla \eta\rangle}{\vert\nabla u\vert^{2}}\Delta_{p}u+(p-2)\vert\nabla u\vert^{p-2}\left\langle\nabla u , \nabla \frac{\langle\nabla u,\nabla \eta\rangle}{\vert\nabla u\vert^{2}}\right\rangle \nonumber \\
&&+(p-2)\vert\nabla u\vert^{p-4} \nabla^{2}{u}(\nabla u, \nabla \eta)+\vert\nabla u\vert^{p-2} \Delta \eta \nonumber\\
&=&\vert\nabla u\vert^{p-2} \Delta \eta+(p-2)\vert\nabla u\vert^{p-4} \nabla^{2}{\eta}(\nabla u, \nabla u)+(p-2) \frac{\langle\nabla u,\nabla \eta\rangle}{\vert\nabla u\vert^{2}}\Delta_{p}u \nonumber \\
&&+2(p-2)\vert\nabla u\vert^{p-4} \nabla^{2}{u}\left(\nabla u, \nabla \eta-\frac{\nabla u}{\vert\nabla u\vert}\langle\frac{\nabla u}{\vert\nabla u\vert},\nabla \eta\rangle\right),\nonumber
\end{eqnarray}
where $\nabla u\neq0$.

Note that $\mathcal{L}_{u}u=(p-1)\Delta_{p}u$. In terms of a local orthonormal frame field $\{e_{1},e_{2},...,e_{n}\}$, we denote by $\mathcal{L}_{u}^{II}$ the second order part of $\mathcal{L}_{u}$, which is
\begin{eqnarray*}
\mathcal{L}^{II}_{u}\eta=\sum_{i,j=1}^{n}(\mathcal{L}^{II}_{u})_{ij}\nabla_{i}\nabla_{j}\eta
\end{eqnarray*}
where
\begin{eqnarray*}
\nabla_{i}=\nabla_{e_{i}},\ \ \ (\mathcal{L}^{II}_{u})_{ij}=\vert\nabla u\vert^{p-2}\delta_{ij}+(p-2)\vert\nabla u\vert^{p-4}\nabla_{i}u\nabla_{j}u,
\end{eqnarray*}
$\delta_{ij}=1$, if $i=j$, otherwise $\delta_{ij}=0$. Note that $\mathcal{L}^{II}_{u}u=\Delta_{p}u$. The principal symbol of this operator $\mathcal{L}_{u}$ is nonnegative everywhere and strictly positive around the points where $\nabla u\neq0$.

Now, we introduce some lemmas which will be used in our proof of the main results. In 2008, Lou \cite{Lou} proved the regularity results of weak solutions for $p$-Laplacian equations in Euclidean spaces. Here, we extend the regularity results to Riemannian manifolds.
\begin{lemma}\label{lemma Lou}
Let $(M,g)$ be an $n$-dimensional complete noncompact Riemannian manifold with nonnegative Ricci curvature. Assume that $\Omega\subset M$ is a bounded and connected domain with boundary $\partial \Omega$ of class $C^{2,\alpha}$. Let $u\in W^{1,p}(\Omega)$ be a local weak solution to the problem \eqref{*} for any $p>1$, then $\vert\nabla u\vert^{p-1}\in W_{\rm loc}^{1,2}(\Omega)$.
\end{lemma}

\begin{proof}
It suffices to prove that $\vert\nabla u\vert^{p-1}\in W_{\rm loc}^{1,2}(B)$ for any ball $B\subset\subset\Omega$. By a similar proof of Theorem 8.17 in \cite{GT}, one can prove that $u\in L_{\rm loc}^{\infty}(\Omega)$ by standard Moser iteration. In particular, $u\in L^{\infty}(\overline{B})$.

Let $\varepsilon\in(0,1)$, $u^{\varepsilon}\in W^{1,p}(B)$ be the unique solution to the following equation:
\begin{equation}\label{Lou1.2}
\begin{cases}
-{\rm div}\left((\varepsilon^{2}+\vert\nabla u^{\varepsilon}\vert^{2})^{\frac{p-2}{2}}\nabla u^{\varepsilon}\right)=1 & \text{in $B$},\\
u^{\varepsilon}=u   & \text{on $\partial B$}.
\end{cases}
\end{equation}
It is not hard to prove that $u^{\varepsilon}\in C^{\infty}(B)$,
\begin{equation}
u^{\varepsilon}\rightarrow u,\ \ {\rm strongly\ in}\ W^{1,p}(B),
\end{equation}
and
\begin{equation}
\int_{B}(\varepsilon^{2}+\vert\nabla u^{\varepsilon}\vert^{2})^{\frac{p}{2}}{\rm dv}\leq C,
\end{equation}
here and hereafter, $C>0$ denotes a constant independent of $\varepsilon$. By De Giorgi estimate, we can get
\begin{equation}
\|u^{\varepsilon}\|_{L_{\infty}}(B)\leq C.
\end{equation}

Now, let
\begin{equation}\label{Lou1.2.1}
G^{*}\equiv(\varepsilon^{2}+\vert\nabla u^{\varepsilon}\vert^{2})^{\frac{p-2}{2}}.
\end{equation}
Rewriting \eqref{Lou1.2} in a local orthonormal frame, we get that
\begin{equation}\label{Lou1.2.2}
\begin{cases}
-\sum\limits_{i=1}^{n}\left(G^{*}u^{\varepsilon}_{i}\right)_{i}=1& \text{in $B$},\\
u^{\varepsilon}=0   & \text{on $\partial B$}.
\end{cases}
\end{equation}
First  we claim that
\begin{eqnarray}\label{Lou1.7}
\sum\limits_{i,k=1}^{n}\left(G^{*}u^{\varepsilon}_{i}\right)_{ki}u_{k}^{\varepsilon}=\sum\limits_{i,k=1}^{n}\left(G^{*}u^{\varepsilon}_{i}\right)_{ik}u^{\varepsilon}_{k}+G^{*}{\rm Ric}(\nabla u^{\varepsilon}, \nabla u^{\varepsilon}),
\end{eqnarray}
where ${\rm Ric}$ denotes the Ricci curvature.

In fact, a straightforward computation shows that
\begin{eqnarray}\label{Lou1.7.1}
\left(G^{*}u^{\varepsilon}_{i}\right)_{ki}u^{\varepsilon}_{k}
=G^{*}_{ki}u^{\varepsilon}_{i}u^{\varepsilon}_{k}+G^{*}_{k}u^{\varepsilon}_{ii}u^{\varepsilon}_{k}+G^{*}_{i}u^{\varepsilon}_{ik}u^{\varepsilon}_{k}+G^{*}u^{\varepsilon}_{iki}u^{\varepsilon}_{k}.
\end{eqnarray}
Since
\begin{eqnarray}\label{Lou1.7.2}
\sum\limits_{i,k=1}^{n}u^{\varepsilon}_{iki}u^{\varepsilon}_{k}=\sum\limits_{i,k=1}^{n}u^{\varepsilon}_{iik}u^{\varepsilon}_{k}+{\rm Ric}(\nabla u^{\varepsilon}, \nabla u^{\varepsilon}).
\end{eqnarray}
Substituting \eqref{Lou1.7.2} into \eqref{Lou1.7.1}, we see that
\begin{eqnarray*}
&&\sum\limits_{i,k=1}^{n}\left(G^{*}u^{\varepsilon}_{i}\right)_{ki}u_{k}^{\varepsilon}\\
&=&\sum\limits_{i,k=1}^{n}\left( G^{*}_{ik}u^{\varepsilon}_{k}u^{\varepsilon}_{i}+G^{*}_{k}u^{\varepsilon}_{ii}u^{\varepsilon}_{k}+G^{*}_{i}u^{\varepsilon}_{ik}u^{\varepsilon}_{k}
+G^{*}u^{\varepsilon}_{iik}u^{\varepsilon}_{k} \right)+G^{*}{\rm Ric}(\nabla u^{\varepsilon}, \nabla u^{\varepsilon})\nonumber\\
&=&\sum\limits_{i,k=1}^{n}\left(G^{*}u^{\varepsilon}_{i}\right)_{ik}u^{\varepsilon}_{k}+G^{*}{\rm Ric}(\nabla u^{\varepsilon}, \nabla u^{\varepsilon}).\nonumber
\end{eqnarray*}
\qquad From \eqref{Lou1.2.2}, we know that
$\sum\limits_{i,k=1}^{n}\left(G^{*}u^{\varepsilon}_{i}\right)_{ik}u^{\varepsilon}_{k}=0$, then \eqref{Lou1.7} becomes
\begin{eqnarray}\label{Lou1.7.4}
\sum\limits_{i,k=1}^{n}\left(G^{*}u^{\varepsilon}_{i}\right)_{ki}u_{k}^{\varepsilon}=G^{*}{\rm Ric}(\nabla u^{\varepsilon}, \nabla u^{\varepsilon}).
\end{eqnarray}

For any domain $\Omega_{0}\subset\subset B$, let $\eta\in C_{c}^{\infty}(\Omega)$ be such that $0\leq \eta\leq1$ in $B$, and $\eta=1$ in $\Omega_{0}$. We note that the key step of Lou \cite{Lou} in proving the regularity results is to prove the integral inequality 2.14 in \cite[p525]{Lou}. Inspired by Lou's nice results, we obtain the following similar integral inequalities for Riemannian manifolds with nonnegative Ricci curvature:
\begin{eqnarray}\label{Lou1.14.3}
\int_{B}\eta^{2}(\varepsilon^{2}+\vert\nabla u^{\varepsilon}\vert^{2})^{p-2}\|\nabla^{2}u^{\varepsilon}\|^{2}{\rm dv}\leq C\int_{B}(\varepsilon^{2}+\vert\nabla u^{\varepsilon}\vert^{2})^{p-1}\vert\nabla\eta\vert^{2}{\rm dv}.
\end{eqnarray}

Now, we will prove the above inequality.  Multiply \eqref{Lou1.7.4} by
$\eta^{2}(\varepsilon^{2}+\vert\nabla u^{\varepsilon}\vert^{2})^{\frac{p-2}{2}}$, and integrate in $B$, then we get that
\begin{eqnarray}\label{Lou1.14}
&&-\int_{B}\eta^{2}(\varepsilon^{2}+\vert\nabla u^{\varepsilon}\vert^{2})^{p-2}{\rm Ric}(\nabla u^{\varepsilon}, \nabla u^{\varepsilon}){\rm dv} \\
&=&-\int_{B}\sum_{i,k=1}^{n}\eta^{2}(\varepsilon^{2}+\vert\nabla u^{\varepsilon}\vert^{2})^{\frac{p-2}{2}}\left(G^{*}u^{\varepsilon}_{i}\right)_{ki}u_{k}^{\varepsilon}
{\rm dv}\nonumber \\
&=&\int_{B}\sum_{i,k=1}^{n}\left(\eta^{2}(\varepsilon^{2}+\vert\nabla u^{\varepsilon}\vert^{2})^{\frac{p-2}{2}}u_{k}^{\varepsilon}\right)_{i}
\left((\varepsilon^{2}+\vert\nabla u^{\varepsilon}\vert^{2})^{\frac{p-2}{2}}u_{i}^{\varepsilon}\right)_{k}{\rm dv}\nonumber\\
&=&\int_{B}2(p-2)\eta(\varepsilon^{2}+\vert\nabla u^{\varepsilon}\vert^{2})^{p-3}\nabla^{2}u^{\varepsilon}(\nabla u^{\varepsilon}, \nabla u^{\varepsilon})\langle\nabla u^{\varepsilon},\nabla \eta\rangle{\rm dv}\nonumber\\
&&+\int_{B}\eta^{2}\left\langle\nabla\left((\varepsilon^{2}+\vert\nabla u^{\varepsilon}\vert^{2})^{\frac{p-2}{2}}\right), \nabla u^{\varepsilon}\right\rangle^{2}{\rm dv}\nonumber \\
&&+\int_{B}2(p-2)\eta^{2}(\varepsilon^{2}+\vert\nabla u^{\varepsilon}\vert^{2})^{p-3}\vert\nabla^{2}u^{\varepsilon}\nabla u^{\varepsilon}\vert^{2}{\rm dv}\nonumber \\
&&+\int_{B}2\eta(\varepsilon^{2}+\vert\nabla u^{\varepsilon}\vert^{2})^{p-2}\nabla^{2}u^{\varepsilon}(\nabla u^{\varepsilon},\nabla\eta){\rm dv}\nonumber \\
&&+\int_{B}\eta^{2}(\varepsilon^{2}+\vert\nabla u^{\varepsilon}\vert^{2})^{p-2}\|\nabla^{2}u^{\varepsilon}\|^{2}{\rm dv}\nonumber\\
&\equiv&{\rm I_{1}+I_{2}+I_{3}+I_{4}+I_{5}},\nonumber
\end{eqnarray}
where the second equality is obtained by using the divergence theorem.

Now, we divide the proof of inequality \eqref{Lou1.14.3} into three cases.

Case ${\rm I}$. When $p>2$. By the Cauchy Schwarz inequality and Young inequality, we deduce that
\begin{eqnarray}\label{Lou1.14.1}
{\rm I_{1}}&\geq&-\int_{B}2(p-2)\eta(\varepsilon^{2}+\vert\nabla u^{\varepsilon}\vert^{2})^{p-3}\vert\nabla^{2}u^{\varepsilon}\nabla u^{\varepsilon}\vert\vert\nabla u^{\varepsilon}\vert^{2}\vert\nabla\eta\vert{\rm dv}\\
&\geq&-\int_{B}(p-2)\eta^{2}(\varepsilon^{2}+\vert\nabla u^{\varepsilon}\vert^{2})^{p-3}\vert\nabla^{2}u^{\varepsilon}\nabla u^{\varepsilon}\vert^{2}{\rm dv}\nonumber\\
&&-\int_{B}(p-2)(\varepsilon^{2}+\vert\nabla u^{\varepsilon}\vert^{2})^{p-1}\vert\nabla\eta\vert^{2}{\rm dv},\nonumber
\end{eqnarray}
and
\begin{eqnarray}\label{Lou1.14.2}
{\rm I_{4}}
&\geq&-\int_{B}2\eta(\varepsilon^{2}+\vert\nabla u^{\varepsilon}\vert^{2})^{p-2}\vert\nabla^{2}u^{\varepsilon}\nabla u^{\varepsilon}\vert\vert\nabla\eta\vert{\rm dv}\\
&\geq&-\int_{B}(p-2)\eta^{2}(\varepsilon^{2}+\vert\nabla u^{\varepsilon}\vert^{2})^{p-3}\vert\nabla^{2}u^{\varepsilon}\nabla u^{\varepsilon}\vert^{2}{\rm dv}\nonumber\\
&&-\int_{B}\frac{1}{p-2}(\varepsilon^{2}+\vert\nabla u^{\varepsilon}\vert^{2})^{p-1}\vert\nabla\eta\vert^{2}{\rm dv}.\nonumber
\end{eqnarray}
Since the Ricci curvature is assumed to be nonnegative and $I_{2}\geq 0$. Then from \eqref{Lou1.14}, \eqref{Lou1.14.1} and \eqref{Lou1.14.2}, we know that the inequality \eqref{Lou1.14.3} holds.

Case ${\rm II}$. When $p=2$. The inequality \eqref{Lou1.14.3} is obvious holds.

Case ${\rm III}$. When $1<p<2$. Choose a local orthonormal frame $\{e_{i}\}$ near any such given point so that at the given point $\nabla u^{\varepsilon}=\vert \nabla u^{\varepsilon} \vert e_{1}$. Then
\begin{eqnarray}\label{Lou1.15.1}
{\rm I_{2}}&=&\int_{B}\eta^{2}\left\langle\nabla\left((\varepsilon^{2}+\vert\nabla u^{\varepsilon}\vert^{2})^{\frac{p-2}{2}}\right), \nabla u^{\varepsilon}\right\rangle^{2}{\rm dv}\\
&=&\int_{B}(p-2)^{2}\eta^{2}(\varepsilon^{2}+\vert\nabla u^{\varepsilon}\vert^{2})^{p-4}\left(\nabla^{2}u^{\varepsilon}(\nabla u^{\varepsilon}, \nabla u^{\varepsilon})\right)^{2}{\rm dv}\nonumber\\
&=&\int_{B}(p-2)^{2}\eta^{2}\left(\varepsilon^{2}+(u_{1}^{\varepsilon})^{2}\right)^{p-4}(u_{1}^{\varepsilon})^{4}(u_{11}^{\varepsilon})^{2}{\rm dv},\nonumber
\end{eqnarray}
and
\begin{eqnarray}\label{Lou1.15.2}
{\rm I_{3}}
&=&\int_{B}2(p-2)\eta^{2}\left(\varepsilon^{2}+(u_{1}^{\varepsilon})^{2}\right)^{p-3}(u_{1}^{\varepsilon})^{2}\sum\limits_{i=1}^{n}(u_{1i}^{\varepsilon})^{2}{\rm dv}\\
&=&\int_{B}2(p-2)\eta^{2}\left(\varepsilon^{2}+(u_{1}^{\varepsilon})^{2}\right)^{p-3}(u_{1}^{\varepsilon})^{2}
\left((u_{11}^{\varepsilon})^{2}+\sum\limits_{i=2}^{n}(u_{1i}^{\varepsilon})^{2}\right){\rm dv}\nonumber\\
&\geq&\int_{B}2(p-2)\eta^{2}\left(\varepsilon^{2}+(u_{1}^{\varepsilon})^{2}\right)^{p-3}(u_{1}^{\varepsilon})^{2}
(u_{11}^{\varepsilon})^{2}{\rm dv}\nonumber\\
&&+\int_{B}2(p-2)\eta^{2}\left(\varepsilon^{2}+(u_{1}^{\varepsilon})^{2}\right)^{p-2}
\sum\limits_{i=2}^{n}(u_{1i}^{\varepsilon})^{2}{\rm dv},\nonumber
\end{eqnarray}
and
\begin{eqnarray}\label{Lou1.15.3}
{\rm I_{5}}&=&\int_{B}\eta^{2}\left(\varepsilon^{2}+(u_{1}^{\varepsilon})^{2}\right)^{p-2}\sum\limits_{i,k=1}^{n}(u_{ik}^{\varepsilon})^{2}{\rm dv}\\
&=&\int_{B}\eta^{2}\left(\varepsilon^{2}+(u_{1}^{\varepsilon})^{2}\right)^{p-2}\left((u_{11}^{\varepsilon})^{2}+2\sum\limits_{i=1}^{n}(u_{1i}^{\varepsilon})^{2}+\sum\limits_{i,k=2}^{n}(u_{ik}^{\varepsilon})^{2}\right){\rm dv}.\nonumber
\end{eqnarray}
Combining \eqref{Lou1.15.1}, \eqref{Lou1.15.2} and \eqref{Lou1.15.3}, we deduce that
\begin{eqnarray}\label{Lou1.15.4}
&&{\rm I_{2}+I_{3}+I_{5}}\\
&\geq&\int_{B}{\rm Q}\eta^{2}\left(\varepsilon^{2}+(u_{1}^{\varepsilon})^{2}\right)^{p-2}(u_{11}^{\varepsilon})^{2}{\rm dv}\nonumber\\
&&+\int_{B}\eta^{2}\left(\varepsilon^{2}+(u_{1}^{\varepsilon})^{2}\right)^{p-2}\left(2(p-1)\sum\limits_{i=1}^{n}(u_{1i}^{\varepsilon})^{2}+\sum\limits_{i,k=2}^{n}(u_{ik}^{\varepsilon})^{2}\right){\rm dv},\nonumber
\end{eqnarray}
where
\begin{eqnarray}\label{Lou1.15.5}
{\rm Q}&=&(p-2)^{2}\frac{(u_{1}^{\varepsilon})^{4}}{\left(\varepsilon^{2}+(u_{1}^{\varepsilon})^{2}\right)^{2}}
+2(p-2)\frac{(u_{1}^{\varepsilon})^{2}}{\varepsilon^{2}+(u_{1}^{\varepsilon})^{2}}+1\\
&=&\left((p-2)\frac{(u_{1}^{\varepsilon})^{2}}{\varepsilon^{2}+(u_{1}^{\varepsilon})^{2}}+1\right)^{2}\nonumber\\
&=&\left(\frac{\varepsilon^{2}+(p-1)(u_{1}^{\varepsilon})^{2}}{\varepsilon^{2}+(u_{1}^{\varepsilon})^{2}}\right)^{2}\nonumber\\
&\geq&(p-1)^{2}.\nonumber
\end{eqnarray}
Substituting \eqref{Lou1.15.5} into \eqref{Lou1.15.4},  we obtain that
\begin{eqnarray}\label{Lou1.15.6}
&&{\rm I_{2}+I_{3}+I_{5}}\\
&\geq&\int_{B}\eta^{2}\left(\varepsilon^{2}+(u_{1}^{\varepsilon})^{2}\right)^{p-2}
\left((p-1)^{2}(u_{11}^{\varepsilon})^{2}+2(p-1)\sum\limits_{i=1}^{n}(u_{1i}^{\varepsilon})^{2}+\sum\limits_{i,k=2}^{n}(u_{ik}^{\varepsilon})^{2}\right){\rm dv}\nonumber\\
&\geq&(p-1)^{2}\int_{B}\eta^{2}\left(\varepsilon^{2}+(u_{1}^{\varepsilon})^{2}\right)^{p-2}
\left((u_{11}^{\varepsilon})^{2}+2\sum\limits_{i=1}^{n}(u_{1i}^{\varepsilon})^{2}+\sum\limits_{i,k=2}^{n}(u_{ik}^{\varepsilon})^{2}\right){\rm dv}\nonumber\\
&=&(p-1)^{2}\int_{B}\eta^{2}\left(\varepsilon^{2}+\vert\nabla u^{\varepsilon}\vert^{2}\right)^{p-2}\|\nabla^{2}u^{\varepsilon}\|^{2}{\rm dv}.\nonumber
\end{eqnarray}

Applying the Cauchy Schwarz inequality and Young inequality, we have that
\begin{eqnarray}\label{Lou1.16}
{\rm I_{1}}&\geq&-\int_{B}2\vert p-2\vert\eta(\varepsilon^{2}+\vert\nabla u^{\varepsilon}\vert^{2})^{p-3}\|\nabla^{2}u^{\varepsilon}\|\vert\nabla u^{\varepsilon}\vert^{3}\vert\nabla\eta\vert{\rm dv}\\
&\geq&-\int_{B}\alpha\eta^{2}(\varepsilon^{2}+\vert\nabla u^{\varepsilon}\vert^{2})^{p-2}\|\nabla^{2}u^{\varepsilon}\|^{2}{\rm dv}\nonumber\\
&&-\int_{B}\alpha^{-1}(p-2)^{2}(\varepsilon^{2}+\vert\nabla u^{\varepsilon}\vert^{2})^{p-1}\vert\nabla\eta\vert^{2}{\rm dv},\nonumber
\end{eqnarray}
and
\begin{eqnarray}\label{Lou1.17}
{\rm I_{4}}
&\geq&-\int_{B}2\eta(\varepsilon^{2}+\vert\nabla u^{\varepsilon}\vert^{2})^{p-2}\|\nabla^{2}u^{\varepsilon}\|\vert\nabla u^{\varepsilon}\vert\vert\nabla\eta\vert{\rm dv}\\
&\geq&-\int_{B}\beta\eta^{2}(\varepsilon^{2}+\vert\nabla u^{\varepsilon}\vert^{2})^{p-2}\|\nabla^{2}u^{\varepsilon}\|^{2}{\rm dv}\nonumber\\
&&-\int_{B}\beta^{-1}(\varepsilon^{2}+\vert\nabla u^{\varepsilon}\vert^{2})^{p-1}\vert\nabla\eta\vert^{2}{\rm dv},\nonumber
\end{eqnarray}
here $\alpha>0$ and $\beta>0$ are constant to be choose later. Since the Ricci curvature is assumed to be nonnegative, and combining \eqref{Lou1.14}, \eqref{Lou1.15.6}, \eqref{Lou1.16}, \eqref{Lou1.17}, and take $\alpha=\beta=\frac{(p-1)^{2}}{4}$, we know that the inequality \eqref{Lou1.14.3} holds.

The remaining part of the proof of this lemma is similar to Lou's proof of  \cite[Lemma 2.1]{Lou}.
\end{proof}

\begin{lemma}\label{LuP Lemma}
Let $(M,g)$ be an $n$-dimensional complete noncompact Riemannian manifold. Assume that $\Omega\subset M$ is a bounded and connected domain with boundary $\partial \Omega$ of class $C^{2,\alpha}$.  For any $u\in C^{2}(\Omega)$, then
\begin{eqnarray}\label{LuP1}
\mathcal{L}_{u}P&=&(p-1)\vert\nabla u\vert^{2(p-2)}\Big(\vert\nabla u\vert^{2-p}\langle\nabla\Delta_{p}u,\nabla u\rangle+\|\nabla^{2}u\|^{2}+(p-2)^{2}A_{u}^{2} \\
&&+{\rm Ric}(\nabla u,\nabla u)\Big)+2(p-1)(p-2)\vert\nabla u\vert^{2(p-2)}\vert\nabla\vert\nabla u\vert\vert^{2}+(p-1)\frac{\Delta_{p}u}{n},\nonumber
\end{eqnarray}
in $\Omega\backslash\mathcal{C}$, where $\mathcal{C}=\{x\in\Omega:\ \vert\nabla u(x)\vert=0\}$.
\end{lemma}

\begin{proof} The main property enjoyed by the linearized $p$-Laplacian is the following version of the celebrated $p$-Bochner formula, also see \cite{NV,VD}.
\begin{eqnarray}\label{7.5}
\frac{1}{p} \mathcal{L}_{u}^{II}\left(\vert\nabla u\vert^{p}\right)&=&\vert\nabla u\vert^{2(p-2)}\Big(\vert\nabla u\vert^{2-p}\big(\langle\nabla \Delta_{p} u, \nabla u\rangle-(p-2) A_{u} \Delta_{p} u\big) \\
&&+\|\nabla^{2}{u}\|^{2}+p(p-2) A_{u}^{2}+{\rm Ric}(\nabla u, \nabla u)\Big).\nonumber
\end{eqnarray}
It follows from \eqref{P-function} and \eqref{7.1}, we have that
\begin{eqnarray}\label{LuP2}
\mathcal{L}_{u}P&=&\vert\nabla u\vert^{p-2} \Delta P+(p-2)\vert\nabla u\vert^{p-4} \nabla^{2}{P}(\nabla u, \nabla u)+(p-2) \frac{\langle\nabla u,\nabla P\rangle}{\vert\nabla u\vert^{2}}\Delta_{p}u \nonumber \\
&&+2(p-2)\vert\nabla u\vert^{p-4} \nabla^{2}{u}\left(\nabla u, \nabla P-\frac{\nabla u}{\vert\nabla u\vert}\left\langle\frac{\nabla u}{\vert\nabla u\vert},\nabla P\right\rangle\right)\nonumber \\
&=&\frac{p-1}{p}\mathcal{L}^{II}_{u}\vert\nabla u\vert^{p}+\frac{1}{n}\mathcal{L}_{u}^{II}u +(p-2) \frac{\langle\nabla u,\nabla P\rangle}{\vert\nabla u\vert^{2}}\Delta_{p}u \\
&&+2(p-2)\vert\nabla u\vert^{p-4} \nabla^{2}{u}\left(\nabla u, \nabla P-\frac{\nabla u}{\vert\nabla u\vert}\left\langle\frac{\nabla u}{\vert\nabla u\vert},\nabla P\right\rangle\right).\nonumber
\end{eqnarray}
Since
\begin{eqnarray}\label{LuP4}
\frac{\langle\nabla u,\nabla P\rangle}{\vert\nabla u\vert^{2}}\Delta_{p}u=(p-1)\vert\nabla u\vert^{p-2}A_{u}\Delta_{p}u+\frac{1}{n}\Delta_{p}u,
\end{eqnarray}
and
\begin{eqnarray}\label{LuP5}
&&\vert\nabla u\vert^{p-4} \nabla^{2}{u}\left(\nabla u, \nabla P-\frac{\nabla u}{\vert\nabla u\vert}\left\langle\frac{\nabla u}{\vert\nabla u\vert},\nabla P\right\rangle\right) \\
&=&(p-1)\vert\nabla u\vert^{2(p-2)}\vert\nabla \vert\nabla u\vert\vert^{2}-(p-1)\vert\nabla u\vert^{2(p-2)}A_{u}^{2}.\nonumber
\end{eqnarray}
Substituting \eqref{7.5}, \eqref{LuP4} and \eqref{LuP5} into \eqref{LuP2}, these complete the proof.
\end{proof}

The main idea of the proof of the following Lemma \ref{More Precise Lemma} comes from \cite{VD}, see Valtorta's Lemma 3.2.

\begin{lemma}\label{More Precise Lemma}
Let $(M,g)$ be an $n$-dimensional complete  Riemannian manifold. Assume that $\Omega\subset M$ is a bounded and connected domain with boundary $\partial \Omega$ of class $C^{2,\alpha}$.  For any $u\in C^{2}(\Omega)$, then \eqref{MPL1} holds in $\Omega\backslash\mathcal{C}$, where $\mathcal{C}=\{x\in\Omega:\ \vert\nabla u(x)\vert=0\}$.
\end{lemma}

\begin{proof} Now, choose a local orthonormal frame $\{e_{i}\}$ near any such given point so that at the given point $\nabla u^{\varepsilon}=\vert \nabla u^{\varepsilon} \vert e_{1}$. Then
\begin{equation*}
\vert\nabla u\vert^{2-p} \Delta_{p}u=\Delta u+(p-2)A_{u}=(p-1) u_{11}+\sum_{i=2}^{n} u_{i i}.
\end{equation*}
By the Cauchy Schwarz inequality, we get that
\begin{eqnarray}\label{MPL3}
&&\|\nabla^{2}{u}\|^{2}+(p^{2}-2p+2)A_{u}^{2}\\
&=&(p^{2}-2p+3)u_{11}^{2}+2 \sum_{j=2}^{n} u_{1 j}^{2}+\sum_{i, j=2}^{n} u_{i j}^{2} \nonumber \\
&\geq&(p-1)^{2} u_{11}^{2}+\frac{1}{n-1}\left(\sum_{i=2}^{n} u_{i i}\right)^{2}+2 \sum_{j=1}^{n} u_{1 j}^{2}. \nonumber
\end{eqnarray}
On the other hand it is easily seen that
\begin{eqnarray}\label{MPL4}
&&\frac{\left(\vert\nabla u\vert^{2-p} \Delta_{p}u\right)^{2}}{n}+\frac{n}{n-1}\left(\frac{\vert\nabla u\vert^{2-p} \Delta_{p}u}{n}-(p-1) A_{u}\right)^{2}+2\vert\nabla\vert\nabla u\vert\vert^{2} \\
&=&\frac{1}{n}\left((p-1) u_{11}+\sum_{i=2}^{n} u_{i i}\right)^{2}+\frac{n}{n-1}\left(-\frac{n-1}{n}(p-1) u_{11}+\frac{1}{n} \sum_{i=2}^{n} u_{i i}\right)^{2}\nonumber \\
&&+2\sum_{j=1}^{n} u_{1 j}^{2} \nonumber \\
&=&(p-1)^{2} u_{11}^{2}+\frac{1}{n-1}\left(\sum_{i=2}^{n} u_{i i}\right)^{2}+2\sum_{j=1}^{n} u_{1 j}^{2}.\nonumber
\end{eqnarray}
These completes the proof.
\end{proof}

Next, we  prove that $P$-function satisfies subharmonic type property of the linearized operator $\mathcal{L}_{u}$.
\begin{lemma}\label{LuP>0 Lemma}
Let $(M,g)$ be an $n$-dimensional complete noncompact Riemannian manifold with nonnegative Ricci curvature. Assume that $\Omega\subset M$ is a bounded and connected domain with boundary $\partial \Omega$ of class $C^{2,\alpha}$. Let $u\in W^{1,p}(\Omega)$ be a weak solution to the problem \eqref{*} for any $p>1$. Then
\begin{eqnarray}\label{10.0}
\mathcal{L}_{u}P\geq0,
\end{eqnarray}
in $\Omega\backslash\mathcal{C}$, where $\mathcal{C}=\{x\in\Omega:\ \vert\nabla u(x)\vert=0\}$.
\end{lemma}

\begin{proof} According to the standard elliptic regularity theory, we know that $u\in C^{3}(\Omega\backslash \mathcal{C})$. From \eqref{*} and Lemma \ref{LuP Lemma}, we see that
\begin{eqnarray}\label{10.1}
\mathcal{L}_{u}P&=&(p-1)\vert\nabla u\vert^{2(p-2)}\Big(\|\nabla^{2}u\|^{2}+(p-2)^{2}A_{u}^{2}+{\rm Ric}(\nabla u,\nabla u)\Big) \\
&&+2(p-1)(p-2)\vert\nabla u\vert^{2(p-2)}\vert\nabla\vert\nabla u\vert\vert^{2}-\frac{p-1}{n}\nonumber \\
&\geq&(p-1)\vert\nabla u\vert^{2(p-2)}\Big(\|\nabla^{2}u\|^{2}+(p^{2}-2p+2)A^{2}_{u}\Big)-2(p-1)^{2}\vert\nabla u\vert^{2(p-2)}A^{2}_{u}\nonumber \\
&&+2(p-1)(p-2)\vert\nabla u\vert^{2(p-2)}\vert\nabla\vert\nabla u\vert\vert^{2}-\frac{p-1}{n}.\nonumber
\end{eqnarray}
By Lemma \ref{More Precise Lemma}, we obtain that
\begin{eqnarray}\label{10.3}
\mathcal{L}_{u}P&\geq& \frac{n(p-1)}{n-1}\left(\frac{1}{n}+(p-1)\vert\nabla u\vert^{p-2} A_{u}\right)^{2}\\
&&+2(p-1)^{2}\vert\nabla u\vert^{2(p-2)}(\vert\nabla\vert\nabla u\vert\vert^{2}-A^{2}_{u}).\nonumber
\end{eqnarray}
From the Cauchy Schwartz inequality, we know that
\begin{eqnarray}\label{10.4}
A_{u}^{2}=\left(\frac{\nabla^{2}u(\nabla u,\nabla u)}{\vert\nabla u\vert^{2}}\right)^{2}\leq\vert\nabla \vert\nabla u\vert\vert^{2}.
\end{eqnarray}
Substituting \eqref{10.4} into \eqref{10.3}, these complete the proof.
\end{proof}

Finally, we state a very fine version of the divergence theorem. This Lemma will be used in the proof of Theorem \ref{Fundamental Identity Theorem}.
\begin{lemma}\label{divergence theorem}
Let $\Omega\subset M$ be a bounded domain with a $C^{2}$-boundary $\partial \Omega$. Assume that $\mathbf{a}: \overline{\Omega}\rightarrow \mathbb{R}^{n}$ satisfies $\mathbf{a}\in[C^{0}(\overline{\Omega})]^{n}$ and ${\rm{div}}\mathbf{a}=f\in L^{1}(\Omega)$ in the sense of distributions in $\Omega$. Then we have
\begin{eqnarray}
\int_{\partial\Omega}\mathbf{a}\cdot\nu {\rm ds}=\int_{\Omega}f(x){\rm dv}.
\end{eqnarray}
\end{lemma}
\begin{remark}
 When $M=\mathbb{R}^{n}$,  the above Lemma is a result of Cuesta and Tak$\acute{a}\check{c}$\cite[Lemma A.1.]{CK2000}. This Lemma is still valid for Riemannian manifolds. The proof is similar to that of Cuesta and Tak$\acute{a}\check{c}$, we refer readers to \cite[page 21-22]{CK2000}.
\end{remark}

\section{Proof of the main results}\label{section3}
\text{\emph{Proof\ of\ Theorem}\ \ref{Fundamental Identity Theorem}}.
Let $u$ be a weak solution to the problem \eqref{*} for any $p>1$.  We choose the vector field $\mathbf{a}$ as follows.
\begin{eqnarray}
\mathbf{a}:=(p-2)\vert\nabla u\vert^{p-4}\langle\nabla u,\nabla P\rangle \nabla u+\vert\nabla u\vert^{p-2} \nabla P.
\end{eqnarray}
It is well known that $u\in C^{2}(\Omega\setminus \mathcal{C})$. By a classical result of G. Stampacchia\cite{GT} ( see Lemma 7.7), the second derivatives of $u$ vanish almost everywhere on the critical set \ $\mathcal{C}$. Then we define that $\mathbf{a}=0\ \ {\rm on} \ \mathcal{C}$.  From the regularity of Lemma \ref{lemma Lou} and Corollary 1.1 in \cite{Lou}, we see that the critical set of $u$ has zero measure. Thus we know that $\int_{\Omega}\mathcal{L}_{u}P{\rm dv}=\int_{\Omega}{\rm div}\mathbf{a}\ {\rm dv}\in L^{1}(\Omega)$ in the sense of distributions in $\Omega$. On the other hand,  $\vert\nabla u\vert\neq0$ on $\partial\Omega$. Then we have that $P_{\nu}=u_{\nu}\left((p-1)\vert u_{\nu}\vert^{p-2}u_{\nu\nu}+\frac{1}{n}\right)\ \text{on}\ \partial\Omega$. So
by Lemma \ref{divergence theorem}, we obtain that
\begin{eqnarray}\label{8.2}
\int_{\Omega}\mathcal{L}_{u}P{\rm dv}&=&\int_{\Omega}{\rm div}\mathbf{a}\ {\rm dv} \\
&=&\int_{\partial\Omega}\left((p-2)\vert\nabla u\vert^{p-4}\langle\nabla u,\nabla P\rangle u_{\nu}+\vert\nabla u\vert^{p-2}P_{\nu}\right){\rm ds} \nonumber \\
&=&(p-1)\int_{\partial\Omega}\vert u_{\nu}\vert^{p-2}u_{\nu}\left((p-1)\vert u_{\nu}\vert^{p-2}u_{\nu\nu}+\frac{1}{n}\right){\rm ds}.\nonumber
\end{eqnarray}

It is easy to see that
\begin{eqnarray}\label{9.1}
\int_{\partial\Omega}u_{\nu}\vert u_{\nu}\vert^{p-2}{\rm ds}&=&\int_{\partial\Omega}(\vert \nabla u \vert^{p-2}\nabla u)\cdot \nu {\rm ds} \\
&=&\int_{\Omega}\Delta_{p}u{\rm dv}\nonumber \\
&=&-\vert\Omega\vert.\nonumber
\end{eqnarray}
Substituting \eqref{6.4} and \eqref{9.1} into \eqref{8.2}, we have that
\begin{eqnarray}\label{8.3}
\int_{\Omega}\mathcal{L}_{u}P{\rm dv}&=&-(p-1)(n-1)\int_{\partial\Omega}\left(\frac{1}{n}\vert u_{\nu}\vert^{p-2}u_{\nu}+H\vert u_{\nu}\vert^{2p-2}\right){\rm ds} \\
&=&(p-1)(n-1)\left(\frac{1}{n}\vert\Omega\vert-\int_{\partial\Omega}H\vert u_{\nu}\vert^{2p-2}{\rm ds}\right).\nonumber
\end{eqnarray}
These complete the proof.
$\hfill\square$

\text{\emph{Proof\ of\ Theorem}\ \ref{Soap Bubble Theorem and the Heintze-Karcher inequality}}.
Let $u$ be a weak solution to the problem \eqref{*} for any $p>1$.
From \eqref{9.1} and Theorem \ref{Fundamental Identity Theorem}, we obtain that
\begin{eqnarray}
&&\int_{\partial \Omega}\frac{1}{H}\left(1+nH\vert u_{\nu}\vert^{p-2} u_{\nu}\right)^{2}{\rm ds}\\
&=&\int_{\partial\Omega}\frac{1}{H}{\rm ds}-2n\vert\Omega\vert+n^{2}\int_{\partial\Omega}H\vert u_{\nu}\vert^{2p-2}{\rm ds}\nonumber \\
&=&\int_{\partial \Omega} \frac{1}{H}{\rm ds} -n\vert\Omega\vert-\frac{n^{2}}{(p-1)(n-1)}\int_{\Omega}\mathcal{L}_{u}P{\rm dv}.\nonumber
\end{eqnarray}
These complete the proof.
$\hfill\square$

\text{\emph{Proof\ of\ Corollary}\ \ref{Heintze-Karcher inequality theorem}}.
Let $u$ be a weak solution to the problem \eqref{*} for any $p>1$.
From the proof of the above Theorem 2, we see that the critical set of $u$ has zero measure. By  Lemma \ref{LuP>0 Lemma}, it implies that $\int_{\Omega}\mathcal{L}_{u}P{\rm dv}\geq0.$ Since the mean curvature $H$ is positive on $\partial\Omega$, then according to the integral identity \eqref{SB and HK} it is easy to obtain Heintze-Karcher inequality \eqref{Heintze-Karcher inequality}. For the proof of rigidity result in Corollary \ref{Heintze-Karcher inequality theorem}, we recommend the reader to refer to the paper {\rm \cite{HR}} or {\rm \cite{Ros}}.
$\hfill\square$

\text{\emph{Proof\ of\ Theorem}\ \ref{overdetermined problem theorem}}.
If $M$ is isometric to an Euclidean ball, according to Corollary \ref{Heintze-Karcher inequality theorem}, it is easy to conclude that the right-hand side of the integral identity \eqref{SB and HK} vanishes. This implies that $u_{\nu}\vert u_{\nu}\vert^{p-2}=-\frac{1}{nH}$ holds on $\partial\Omega$.

Now we assume conversely that  $u_{\nu}\vert u_{\nu}\vert^{p-2}=-\frac{1}{nH}$ holds on $\partial\Omega$. From \eqref{*}, \eqref{Fundamental Identity} and \eqref{9.1}, we obtain that
\begin{eqnarray}\label{12.1}
\int_{\Omega}\mathcal{L}_{u}P{\rm dv}&=&\frac{(p-1)(n-1)}{n}\left(\vert\Omega\vert-n\int_{\partial\Omega}H\vert u_{\nu}\vert^{2p-2}{\rm ds}\right)\\
&=&\frac{(p-1)(n-1)}{n}\left(\vert\Omega\vert+\int_{\partial\Omega}\vert u_{\nu}\vert^{p-2}u_{\nu}{\rm ds}\right)\nonumber \\
&=&\frac{(p-1)(n-1)}{n}\left(\vert\Omega\vert+\int_{\Omega}\Delta_{p}u{\rm dv}\right)\nonumber \\
&=&0.\nonumber
\end{eqnarray}
Combining \eqref{SB and HK} and \eqref{12.1}, we know that $\int_{\partial\Omega}\frac{1}{H}{\rm ds}=n\vert\Omega\vert$ holds. According to Corollary \ref{Heintze-Karcher inequality theorem}, these complete the proof.
$\hfill\square$

\text{\emph{Proof\ of\ Theorem}\ \ref{Soap Bubble Theorem}}.
Let $u$ be a weak solution to the problem \eqref{*} for any $p>1$.
From \eqref{9.1}, we have the following identity
\begin{eqnarray}\label{9.2}
\frac{1}{n^{2}H_{0}}\int_{\partial\Omega}(nu_{\nu}\vert u_{\nu}\vert^{p-2}H_{0}+1)^{2}{\rm ds}
=H_{0}\int_{\partial\Omega}\vert u_{\nu}\vert^{2p-2}{\rm ds}-\frac{1}{n}\vert\Omega\vert.
\end{eqnarray}
Hence
\begin{eqnarray}\label{9.3}
\int_{\partial\Omega}H\vert u_{\nu}\vert^{2p-2}{\rm ds}
&=&H_{0}\int_{\partial\Omega}\vert u_{\nu}\vert^{2p-2}{\rm ds}+\int_{\partial\Omega}(H-H_{0})\vert u_{\nu}\vert^{2p-2}{\rm ds} \\
&=&\frac{1}{n^{2}H_{0}}\int_{\partial\Omega}(nu_{\nu}\vert u_{\nu}\vert^{p-2}H_{0}+1)^{2}{\rm ds}+\frac{1}{n}\vert\Omega\vert \nonumber \\
&&+\int_{\partial\Omega}(H-H_{0})\vert u_{\nu}\vert^{2p-2}{\rm ds}.\nonumber
\end{eqnarray}
Substituting \eqref{9.3} into \eqref{8.3}. These complete the proof.
$\hfill\square$

\text{\emph{Proof\ of\ Corollary}\ \ref{overdetermined problem theorem 2}}.
Let $u$ be a weak solution to the problem \eqref{*} for any $p>1$.
If the mean curvature $H=H_{0}$ on $\partial\Omega$, then the right-side of the equality \eqref{SBT} vanishes. Combining Theorem \ref{Soap Bubble Theorem} and Lemma \ref{LuP>0 Lemma}, we deduce that $\vert u_{\nu}\vert^{p-2}u_{\nu}=-\frac{1}{nH}$ on $\partial\Omega$, and then according to Theorem \ref{overdetermined problem theorem}, these complete the proof.
$\hfill\square$

\def\cprime{$'$}


\begin{thebibliography}{10}

\bibitem{A12}
A.D. Alexandrov.
\newblock {A characteristic property of spheres}.
\newblock {\em Ann. Math. Pura Appl.} $\mathbf{58}$: 303--315, 1962.


\bibitem{ACF}
C.A. Antonini, G. Ciraolo, and A. Farina.
\newblock Interior regularity results for inhomogeneous anisotropic quasilinear equations.
\newblock {\em Mathematische Annalen}. DOI: 10.1007/s00208-022-02500-x (2022)

\bibitem{BQ}
D. Bakry and Z. Qian.
\newblock Some new results on eigenvectors via dimension, diameter, and Ricci curvature.
\newblock {\em Adv. Math}. $\mathbf{155}$(1): 98--153, 2000.

\bibitem{Bianchini and Ciraolo[29]}
C. Bianchini and G. Ciraolo.
\newblock Wulff shape characterizations in overdetermined anisotropic elliptic problems.
\newblock {\em Commun. Part. Diff. Eq}. $\mathbf{43}$ (5): 790--820, 2018.

\bibitem{BH}
F.Brock and A. Henrot.
\newblock A symmetry result for an overdetermined elliptic problem using continuous rearrangement and domain derivative.
\newblock {\em Rend. Circ. Mat. Palermo} $\mathbf{51}$: 375--390, 2002.

\bibitem{Cianchi[18]}
A. Cianchi and P. Salani.
\newblock Overdetermined anisotropic elliptic problems.
\newblock {\em Math. Ann.} $\mathbf{345}$: 859--881, 2009.

\bibitem{GV}
G. Ciraolo and L. Vezzoni.
\newblock On Serrin's overdetermined problem in space forms.
\newblock {\em Manuscripta Math.} $\mathbf{159}$: 445--452, 2019.

\bibitem{CK2000}
M. Cuesta and P. Tak$\acute{a}\check{c}$.
A strong comparison principle for positive solutions of degenerate elliptic equations.
\newblock {\em Differ. Integral. Equ.} $\mathbf{13}$(4-6): 721--746, 2000.

\bibitem{FF}
F. Colasuonno and F. Ferrari.
The Soap Bubble Theorem and a $p$-Laplacian overdetermined problem.
\newblock {\em Commun. Fur. Appl. Anal.} $\mathbf{19}$(2): 983-1000, 2020.

\bibitem{DP}
L. Damascelli and F. Pacella.
Monotonicity and symmetry results for $p$-Laplace equations and applicalitions.
\newblock {\em Adv. Difer. Equ.} $\mathbf{5}$(7-9): 1179--1200, 2000.

\bibitem{DY}
H. Deng and J. Yin.
The overdetermined problem and lower bound estimate on the first nonzero Steklov eigenvalue of $p$-Laplacian.
\newblock {\em Ann. Mat. Pur. Appl.} $\mathbf{201}$: 2037--2053, 2022.

\bibitem{FAW}
M.M. Fall, I.A. Minlend, and T. Weth.
Serrin's overdetermined problem on the sphere.
\newblock {\em Calc. Var. Partial Differ. Equ.} $\mathbf{57}$(1): 3--27, 2018.

\bibitem{AA}
A. Farina and A. Roncoroni.
Serrin's type problems in warped product manifolds.
\newblock {\em Commun. Contemp. Math.} $\mathbf{24}$(4): 1--21, 2022.

\bibitem{GL}
N. Garofalo and J.L. Lewis.
A symmetry result related to some overdetermined boundary value problems.
\newblock {\em Am. J. Math.} $\mathbf{111}$(1): 9--33, 1989.

\bibitem{GT}
D. Gilbarg and N.S. Trudinger.
\newblock {\em Elliptic partial differential equations of second order.} Springer, 2001.

\bibitem{HR}
Q. Huang and Q. Ruan.
Applications of some elliptic equations in Riemannian manifolds.
\newblock {\em J. Math. Anal. Appl.} $\mathbf{409}$: 189--196, 2014.

\bibitem{Lou}
H. Lou.
On singular sets of local solutions to $p$-Laplace equations.
\newblock {\em Chin. Ann. Math.} $\mathbf{29}$B(5): 521--530, 2008.

\bibitem{MP3}
R. Magnanini and G. Poggesi.
On the stability for Alexandrov's Soap Bubble theorem.
\newblock {\em J. Anal. Math.} $\mathbf{139}$(1): 179--205, 2019.

\bibitem{MP2}
R. Magnanini and G. Poggesi.
Serrin's problem and Alexandrov's soap bubble theorem: enhanced stability via integral identities.
\newblock {\em Indiana Univ. Math. Jour.}  $\mathbf{69}$(4): 1181--1205, 2020.

\bibitem{NV}
A. Naber and D. Valtorta.
Sharp estimates on the frst eigenvalue of the $p$-Laplacian with negative Ricci lower bound.
\newblock {\em Math. Z.} $\mathbf{277}$(3--4): 867--891, 2014.

\bibitem{PS}
L.E. Payne and P.W. Schaefer.
Duality theorems in some overdetermined boundary value problems.
\newblock {\em Math. Methods Appl. Sci.} $\mathbf{11}$(6): 805--819, 1989.

\bibitem{Re1}
R.C. Reilly.
Applications of the Hessian operator in a Riemannian manifold.
\newblock {\em Indiana Univ. Math. J.} $\mathbf{26}$: 459--472, 1977.

\bibitem{Re2}
R.C. Reilly.
Mean curvature, the Laplacian, and soap bubbles.
\newblock {\em Amer. Math. Monthly} $\mathbf{89}$: 180--188, 1982.

\bibitem{Ros}
A. Ros.
Compact hypersurfaces with constant higher order mean curvatures.
\newblock {\em Rev. Mat. Iberoam.} $\mathbf{3}$: 447--453, 1987.

\bibitem{Ros and Sicbaldi[2]}
A. Ros and P. Sicbaldi.
Geometry and topology of some overdetermined elliptic problems.
\newblock {\em J. Differential Equations} $\mathbf{255}$: 951--977, 2013.

\bibitem{Serrin[1]}
J. Serrin.
A symmetry problem in potential theory.
\newblock {\em Arch. Ration. Mech. Anal.} $\mathbf{43}$: 304--318, 1971.


\bibitem{S1956}
I.S. Sokolnikoff.
\newblock {\em Mathematical Theory of Elasticity.} McGraw-HillBook Company, Inc., NewYork, Toronto, London, 1956.


\bibitem{VD}
D. Valtorta.
Sharp estimate on the frst eigenvalue of the $p$-Laplacian.
\newblock {\em Nonlinear Anal.} $\mathbf{75}$(13): 4974--4994, 2012.

\bibitem{Wang and Xia[23]}
G. Wang and C. Xia.
A characterization of the Wulff shape by an overdetermined anisotropic PDE.
\newblock {\em Arch. Ration. Mech. Anal.} $\mathbf{199}$(1): 99--115, 2011.

\bibitem{WH}
H. Weinberger.
Remark on the preceding paper of Serrin.
\newblock {\em Arch. Ration. Mech. Anal.} $\mathbf{43}$: 319--320, 1971.



\end{thebibliography}
\end{document}